\documentclass{amsart}
\usepackage{amsfonts}

\newtheorem{theorem}{Theorem}
\theoremstyle{plain}

\newtheorem{definition}{Definition}
\newtheorem{example}{Example}

\numberwithin{equation}{section}

\input{tcilatex}

\begin{document}
\title[The uniqueness of involution in locally C*-algebras]{A note on the
uniqueness of involution in locally C*-algebras}
\author{Alexander A. Katz}
\address{Dr. Alexander A. Katz, Department of Mathematics and Computer
Science, St. John's College of Liberal Arts and Sciences, St. John's
University, 300 Howard Avenue, DaSilva Academic Center 314, Staten Island,
NY 10301, USA}
\email{katza@stjohns.edu}
\date{January 2, 2012}
\thanks{2010 AMS Subject Classification: Primary 46K05}
\keywords{C*-algebras, locally C*-algebras, projective limit of projective
family of C*-algebras}

\begin{abstract}
In the present note we show that the involution in locally C*-algebras is
uniquely determined.
\end{abstract}

\maketitle

\section{Introduction}

One of the important basic facts of the theory of $C^{\ast }$\textbf{%
-algebras} is that the unary operation of involution in a $C^{\ast }$%
-algebra is uniquely determined. This property was first observed in 1955 by
Bohnenblust and Karlin in\textbf{\ }\cite{BohnenblustKarlin1955}\textbf{\ (}%
see as well\textbf{\ }\cite{Rickart1960}\textbf{\ }for a nice exposition%
\textbf{).}

The Hausdorff projective limits of projective families of Banach algebras as
natural locally-convex generalizations of Banach algebras have been studied
sporadically by many authors since 1952, when they were first introduced by
Arens \cite{Arens1952} and Michael \cite{Michael1952}. The Hausdorff
projective limits of projective families of $C^{\ast }$-algebras were first
mentioned by Arens \cite{Arens1952}. They have since been studied under
various names by many authors. Development of the subject is reflected in
the monograph of Fragoulopoulou \cite{Fragoulopoulou2005}. We will follow
Inoue \cite{Inoue1971} in the usage of the name \textbf{locally }$C^{\ast }$%
\textbf{-algebras} for these algebras.

The purpose of the present notes is to show that the unary operation of
involution in locally $C^{\ast }$-algebras is uniquely determined.

\section{Preliminaries}

First, we recall some basic notions on topological $^{\ast }$-algebras. A $%
^{\ast }$-algebra (or involutive algebra) is an algebra $A$ over $%
\mathbb{C}
$ with an involution 
\begin{equation*}
^{\ast }:A\rightarrow A,
\end{equation*}
such that 
\begin{equation*}
(a+\lambda b)^{\ast }=a^{\ast }+\overline{\lambda }b^{\ast },
\end{equation*}
and 
\begin{equation*}
(ab)^{\ast }=b^{\ast }a^{\ast },
\end{equation*}%
for every $a,b\in A$ and $\lambda \in $ $%
\mathbb{C}
$.

A seminorm $\left\Vert .\right\Vert $ on a $^{\ast }$-algebra $A$ is a $%
C^{\ast }$-seminorm if it is submultiplicative, i.e. 
\begin{equation*}
\left\Vert ab\right\Vert \leq \left\Vert a\right\Vert \left\Vert
b\right\Vert ,
\end{equation*}

and satisfies the $C^{\ast }$-condition, i.e. 
\begin{equation*}
\left\Vert a^{\ast }a\right\Vert =\left\Vert a\right\Vert ^{2},
\end{equation*}%
for every $a,b\in A.$ Note that the $C^{\ast }$-condition alone implies that 
$\left\Vert .\right\Vert $ is submultiplicative, and in particular 
\begin{equation*}
\left\Vert a^{\ast }\right\Vert =\left\Vert a\right\Vert ,
\end{equation*}%
for every $a\in A$ (cf. for example \cite{Fragoulopoulou2005}).

When a seminorm $\left\Vert .\right\Vert $ on a $^{\ast }$-algebra $A$ is a $%
C^{\ast }$-norm, and $A$ is complete in in the topology generated by this
norm, $A$ is called a $C^{\ast }$\textbf{-algebra}. The following theorem is
valid.

\begin{theorem}[Bohnenblust and Karlin \protect\cite{BohnenblustKarlin1955}]

The unary operation of involution in a $C^{\ast }$-algebra is uniquely
determined.
\end{theorem}

\begin{proof}
See for example \cite{Rickart1960} for details.
\end{proof}

A topological $^{\ast }$-algebra is a $^{\ast }$-algebra $A$ equipped with a
topology making the operations (addition, multiplication, additive inverse,
involution) jointly continuous. For a topological $^{\ast }$-algebra $A$,
one puts $N(A)$ for the set of continuous $C^{\ast }$-seminorms on $A.$ One
can see that $N(A)$ is a directed set with respect to pointwise ordering,
because 
\begin{equation*}
\max \{\left\Vert .\right\Vert _{\alpha },\left\Vert .\right\Vert _{\beta
}\}\in N(A)
\end{equation*}%
for every $\left\Vert .\right\Vert _{\alpha },\left\Vert .\right\Vert
_{\beta }\in N(A),$ where $\alpha ,\beta \in \Lambda ,$ with $\Lambda $
being a certain directed set.

For a topological $^{\ast }$-algebra $A$, and $\left\Vert .\right\Vert
_{\alpha }\in N(A),$ $\alpha \in \Lambda $, 
\begin{equation*}
\ker \left\Vert .\right\Vert _{\alpha }=\{a\in A:\left\Vert a\right\Vert
_{\alpha }=0\}
\end{equation*}%
is a $^{\ast }$-ideal in $A$, and $\left\Vert .\right\Vert _{\alpha }$
induces a $C^{\ast }$-norm (we as well denote it by $\left\Vert .\right\Vert
_{\alpha }$) on the quotient $A_{\alpha }=A/\ker \left\Vert .\right\Vert
_{\alpha }$, and $A_{\alpha }$ is automatically complete in the topology
generated by the norm $\left\Vert .\right\Vert _{\alpha },$ thus is a $%
C^{\ast }$-algebra (see \cite{Fragoulopoulou2005} for details). Each pair $%
\left\Vert .\right\Vert _{\alpha },\left\Vert .\right\Vert _{\beta }\in N(A),
$ such that 
\begin{equation*}
\beta \succeq \alpha ,
\end{equation*}%
$\alpha ,\beta \in \Lambda ,$ induces a natural (continuous) surjective $%
^{\ast }$-homomorphism 
\begin{equation*}
g_{\alpha }^{\beta }:A_{\beta }\rightarrow A_{\alpha }.
\end{equation*}

Let, again, $\Lambda $ be a set of indices, directed by a relation
(reflexive, transitive, antisymmetric) $"\preceq "$. Let 
\begin{equation*}
\{A_{\alpha },\alpha \in \Lambda \}
\end{equation*}%
be a family of $C^{\ast }$-algebras, and $g_{\alpha }^{\beta }$ be, for 
\begin{equation*}
\alpha \preceq \beta ,
\end{equation*}%
the continuous linear $^{\ast }$-mappings 
\begin{equation*}
g_{\alpha }^{\beta }:A_{\beta }\longrightarrow A_{\alpha },
\end{equation*}%
so that 
\begin{equation*}
g_{\alpha }^{\alpha }(x_{\alpha })=x_{\alpha },
\end{equation*}%
for all $\alpha \in \Lambda ,$ and 
\begin{equation*}
g_{\alpha }^{\beta }\circ g_{\beta }^{\gamma }=g_{\alpha }^{\gamma },
\end{equation*}%
whenever 
\begin{equation*}
\alpha \preceq \beta \preceq \gamma .
\end{equation*}%
\ Let $\Gamma $ be the collections $\{g_{\alpha }^{\beta }\}$ of all such
transformations.\ Let $A$ be a $^{\ast }$-subalgebra of the direct product
algebra%
\begin{equation*}
\dprod\limits_{\alpha \in \Lambda }A_{\alpha },
\end{equation*}%
so that for its elements 
\begin{equation*}
x_{\alpha }=g_{\alpha }^{\beta }(x_{\beta }),
\end{equation*}%
for all%
\begin{equation*}
\alpha \preceq \beta ,
\end{equation*}%
where 
\begin{equation*}
x_{\alpha }\in A_{\alpha },
\end{equation*}%
and%
\begin{equation*}
x_{\beta }\in A_{\beta }.
\end{equation*}

\begin{definition}
The $^{\ast }$-algebra $A$ constructed above is called a \textbf{Hausdorff} 
\textbf{projective limit} of the projective family 
\begin{equation*}
\{A_{\alpha },\alpha \in \Lambda \},
\end{equation*}%
relatively to the collection%
\begin{equation*}
\Gamma =\{g_{\alpha }^{\beta }:\alpha ,\beta \in \Lambda :\alpha \preceq
\beta \},
\end{equation*}%
and is denoted by%
\begin{equation*}
\underleftarrow{\lim }A_{\alpha },
\end{equation*}%
and is called the Arens-Michael decomposition of $A$.
\end{definition}

It is well known (see, for example \cite{Treves1967}) that for each $x\in A,$
and each pair $\alpha ,\beta \in \Lambda ,$ such that $\alpha \preceq \beta ,
$ there is a natural projection 
\begin{equation*}
\pi _{\beta }:A\longrightarrow A_{\beta },
\end{equation*}%
defined by 
\begin{equation*}
\pi _{\alpha }(x)=g_{\alpha }^{\beta }(\pi _{\beta }(x)),
\end{equation*}%
and each projection $\pi _{\alpha }$ for all $\alpha \in \Lambda $ is
continuous.

\begin{definition}
A topological $^{\ast }$-algebra $A$ over $\mathbb{C}$\ \ is called a 
\textbf{locally }$C^{\ast }$\textbf{-algebra} if there exists a projective
family of $C^{\ast }$-algebras 
\begin{equation*}
\{A_{\alpha };g_{\alpha }^{\beta };\alpha ,\beta \in \Lambda \},
\end{equation*}%
so that 
\begin{equation*}
A\cong \underleftarrow{\lim }A_{\alpha },
\end{equation*}%
i.e. $A$ is topologically $^{\ast }$-isomorphic to a projective limit of a
projective family of $C^{\ast }$-algebras, i.e. there exits its
Arens-Michael decomposition of $A$ composed entirely of $C^{\ast }$-algebras.
\end{definition}

A topological $^{\ast }$-algebra $A$ over $\mathbb{C}$ is a locally $C^{\ast
}$-algebra iff $A$ is a complete Hausdorff topological $^{\ast }$-algebra in
which topology is generated by a saturated separating family of $C^{\ast }$%
-seminorms (see \cite{Fragoulopoulou2005} for details).

\begin{example}
Every $C^{\ast }$-algebra is a locally $C^{\ast }$-algebra.
\end{example}

\begin{example}
A closed $^{\ast }$-subalgebra of a locally $C^{\ast }$-algebra is a locally 
$C^{\ast }$-algebra.
\end{example}

\begin{example}
The product $\dprod\limits_{\alpha \in \Lambda }A_{\alpha }$ of $C^{\ast }$%
-algebras $A_{\alpha }$, with the product topology, is a locally $C^{\ast }$%
-algebra.
\end{example}

\begin{example}
Let $X$ be a compactly generated Hausdorff space (this means that a subset $%
Y\subset X$ is closed iff $Y\cap K$ is closed for every compact subset $%
K\subset X$). Then the algebra $C(X)$ of all continuous, not necessarily
bounded complex-valued functions on $X,$ with the topology of uniform
convergence on compact subsets, is a locally $C^{\ast }$-algebra. It is well
known that all metrizable spaces and all locally compact Hausdorff spaces
are compactly generated (see \cite{Kelley1975} for details).
\end{example}

Let $A$ be a locally $C^{\ast }$-algebra. Then an element $a\in A$ is called 
\textbf{bounded}, if 
\begin{equation*}
\left\Vert a\right\Vert _{\infty }=\{\sup \left\Vert a\right\Vert _{\alpha
},\alpha \in \Lambda :\left\Vert .\right\Vert _{\alpha }\in N(A)\}<\infty .
\end{equation*}%
The set of all bounded elements of $A$ is denoted by $b(A).$

It is well-known that for each locally $C^{\ast }$-algebra $A,$ its set $b(A)
$ of bounded elements of $A$ is a locally $C^{\ast }$-subalgebra, which is a 
$C^{\ast }$-algebra in the norm $\left\Vert .\right\Vert _{\infty },$ such
that it is dense in $A$ in its topology (see for example \cite%
{Fragoulopoulou2005}).

\section{The uniqueness of involuton in locally C*-algebras}

Here we present the main theorem of the current notes.

\begin{theorem}
The unary operation of involution in any locally $C^{\ast }$-algebra is
unique, i.e., if $(A,^{\ast },\left\Vert .\right\Vert _{\alpha },\alpha \in
\Lambda )$ and $(A,^{\#},\left\Vert .\right\Vert _{\alpha },\alpha \in
\Lambda )$ are two locally $C^{\ast }$-algebras, means that each seminorm $%
\left\Vert .\right\Vert _{\alpha },\alpha \in \Lambda ,$ satisfies the $%
C^{\ast }$-property for both operations, $"\ast "$ and $"\#"$, then 
\begin{equation*}
^{\ast }=\#
\end{equation*}%
on $A$.
\end{theorem}

\begin{proof}
Let now $A$ be a locally $C^{\ast }$-algebra, and let 
\begin{equation*}
A=\underleftarrow{\lim }A_{\alpha },
\end{equation*}
$\alpha \in \Lambda $, be its Arens-Michael decomposition, built using the
family of seminorms $\left\Vert .\right\Vert _{\alpha },\alpha \in \Lambda ,$
so that for each $\alpha \in \Lambda ,$ 
\begin{equation*}
(A_{\alpha },\ast _{\alpha },\left\Vert .\right\Vert _{\alpha })
\end{equation*}
and 
\begin{equation*}
(A_{\alpha },\#_{\alpha },\left\Vert .\right\Vert _{\alpha })
\end{equation*}
are $C^{\ast }$-algebras, where the unary operations  $"\ast _{\alpha }"$
and $"\#_{\alpha }"$ on $A_{\alpha }$ are defined as follows:

\begin{equation*}
\pi _{\alpha }(x^{\ast })=(\pi _{\alpha }(x))^{\ast _{\alpha }},
\end{equation*}%
and 
\begin{equation*}
\pi _{\alpha }(x^{\#})=(\pi _{\alpha }(x))^{\#_{\alpha }},
\end{equation*}%
for each $x\in A$ and $\alpha \in \Lambda .$

Let us now assume, to the contrary to the statement of the theorem, that
there exists some $x\in A,$ such that 
\begin{equation*}
x^{\ast }=y\neq z=x^{\#}.
\end{equation*}
Then there must exist $\alpha _{0}\in \Lambda ,$ such that 
\begin{equation*}
\pi _{\alpha _{0}}(y)\neq \pi _{\alpha _{0}}(z).
\end{equation*}
In fact, if it is not the case, and 
\begin{equation*}
\pi _{\alpha }(y)=\pi _{\alpha }(z)
\end{equation*}
for each $\alpha \in \Lambda ,$ implies that 
\begin{equation*}
y=z,
\end{equation*}
which contradicts the assumption.

So, $\alpha _{0}$ must be such that  
\begin{equation*}
\pi _{\alpha _{0}}(x^{\ast })\neq \pi _{\alpha _{0}}(x^{\#}),
\end{equation*}
which means that for 
\begin{equation*}
\pi _{\alpha _{0}}(x)=x_{\alpha _{0}}\in A_{\alpha _{0}},
\end{equation*}
\begin{equation*}
x_{\alpha _{0}}^{\ast _{\alpha }}\neq x_{\alpha _{0}}^{\#_{\alpha }},
\end{equation*}
which contradicts Theorem 1. Found contradiction proves the theorem.
\end{proof}


\begin{thebibliography}{9}
\bibitem{Arens1952} \textbf{Arens, R.,} \textit{A generalization of normed
rings.} (English) Pacific J. Math., Vol. 2, 1952, pp. 455--471.

\bibitem{BohnenblustKarlin1955} \textbf{Bohnenblust, H.F.; Karlin, S.,} 
\textit{Geometrical properties of the unit sphere of Banach algebras.}
(English) Ann. of Math. (2) No. 62, 1955, pp. 217--229.

\bibitem{Fragoulopoulou2005} \textbf{Fragoulopoulou, M.,} \textit{%
Topological algebras with involution.} (English) North-Holland Mathematics
Studies, Vol. 200, Elsevier Science B.V., Amsterdam, 2005, 495 pp.

\bibitem{Inoue1971} \textbf{Inoue, A.,} \textit{Locally }$C^{\ast }$\textit{%
-algebra. }(English), Mem. Fac. Sci. Kyushu Univ. Ser. A , Vol. 25, 1971,
pp. 197--235.

\bibitem{Kelley1975} \textbf{Kelley, J.L.,} \textit{General topology.}
(English) Reprint of the 1955 edition [Van Nostrand, Toronto, Ont.].
Graduate Texts in Mathematics, No. 27. Springer-Verlag, New York-Berlin,
1975, 298 pp

\bibitem{Michael1952} \textbf{Michael, E.A.,} \textit{Locally
multiplicatively-convex topological algebras.} (English) Mem. Amer. Math.
Soc., No. 11, 1952, 79 pp.

\bibitem{Rickart1960} \textbf{Rickart, C.E.,} \textit{General theory of
Banach algebras.} (English) The University Series in Higher Mathematics D.
van Nostrand Co., Inc., Princeton, N.J.-Toronto-London-New York, 1960, 394
pp.

\bibitem{Treves1967} \textbf{Tr\`{e}ves, F.,} \textit{Topological vector
spaces: Distributions and Kernels.} (English), New York-London: Academic
Press., 1967, 565 pp.
\end{thebibliography}
\end{document}